\documentclass[12pt]{article}
\usepackage{hyperref}
\usepackage{a4wide}
\usepackage{tikz}
\usepackage{amsmath,amsfonts,amssymb,latexsym,graphics,epsfig,url}
\usepackage{color}
\usepackage{amsthm,enumerate}
\usepackage[english]{babel}
\usepackage{graphicx}
\usepackage{soul}
\usepackage[all]{xy} 

\newcommand\blfootnote[1]{%
	\begingroup
	\renewcommand\thefootnote{}\footnote{#1}%
	\addtocounter{footnote}{-1}%
	\endgroup
}

\newtheorem{theorem}{Theorem}[section]
\newtheorem{proposition}[theorem]{Proposition}

\newtheorem{corollary}[theorem]{Corollary}

\newtheorem{guess}[theorem]{Conjecture}
\newtheorem{example}[theorem]{Example}

\DeclareMathOperator{\Irep}{Irep}

\DeclareMathOperator{\rank}{rank}
\DeclareMathOperator{\Sp}{Sp}

\DeclareMathOperator{\tr}{tr}

\def\C{\mathbb C}

\def\G{\Gamma}

\begin{document}


\title{Spectra and eigenspaces of arbitrary lifts of graphs
	}
\author{C. Dalf\'o\footnote{Departament  de Matem\`atica, Universitat de Lleida, Igualada (Barcelona), Catalonia, {\tt{cristina.dalfo@matematica.udl.cat}}},
	M. A. Fiol\footnote{Departament de Matem\`atiques, Universitat Polit\`ecnica de Catalunya; and Barcelona Graduate School of Mathematics, Barcelona, Catalonia, {\tt{miguel.angel.fiol@upc.edu}}},
	 S. Pavl\'ikov\'a\footnote{Department of Mathematics and Descriptive Geometry,
	 	Slovak University of Technology, Bratislava, Slovak Republic,
	 	{\tt {sona.pavlikova@stuba.sk}}}, and
	J. \v{S}ir\'a\v{n}\footnote{Department of Mathematics and Statistics,
		The Open University, Milton Keynes, UK; and
		Department of Mathematics and Descriptive Geometry,
		Slovak University of Technology, Bratislava, Slovak Republic,
		{\tt {j.siran@open.ac.uk}}}
}
\date{}

\maketitle


\begin{abstract}
We describe, in a very explicit way, a method for determining the spectra and bases of all the corresponding eigenspaces of arbitrary lifts of graphs (regular or not).
\end{abstract}



\blfootnote{\begin{minipage}[l]{0.3\textwidth} \includegraphics[trim=10cm 6cm 10cm 5cm,clip,scale=0.15]{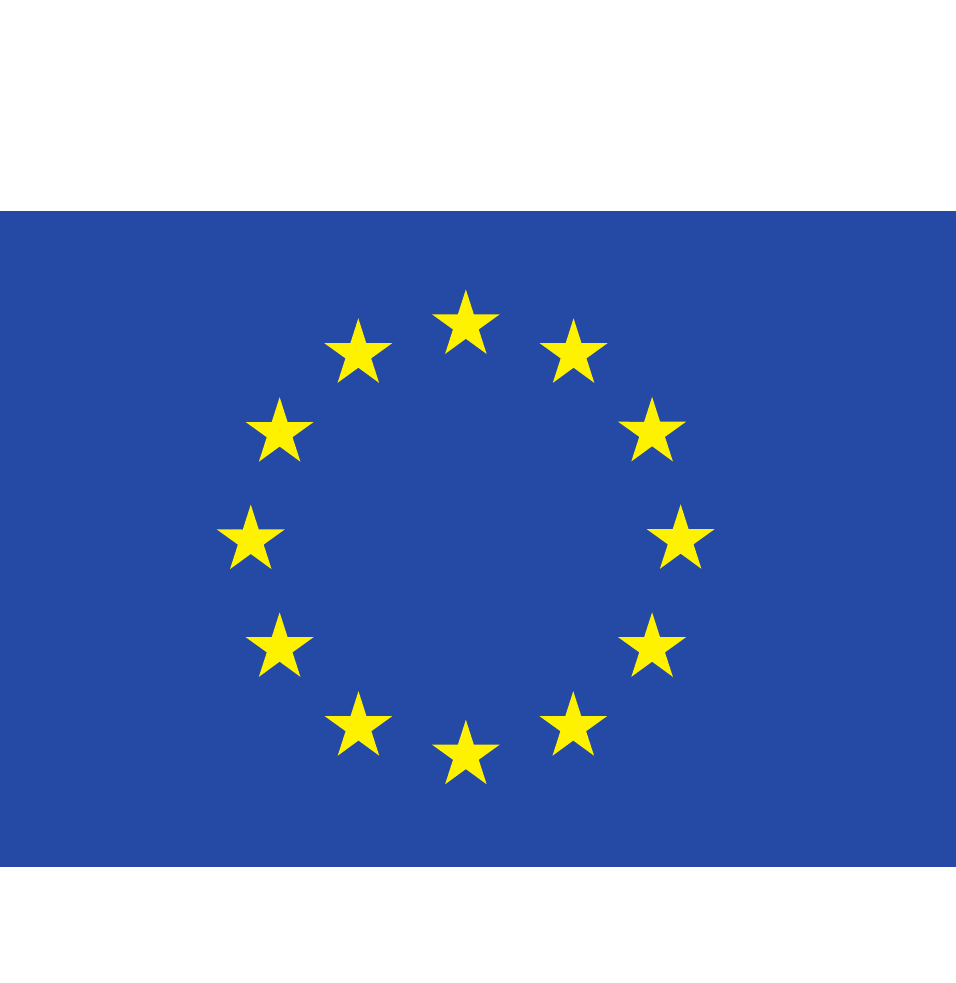} \end{minipage}  \hspace{-2cm} \begin{minipage}[l][1cm]{0.79\textwidth}
		The first author has received funding from the European Union's Horizon 2020 research and innovation programme under the Marie Sk\l{}odowska-Curie grant agreement No 734922.
\end{minipage}}

\section{Introduction}\label{sec:int}

For a graph $\Gamma$ with adjacency matrix $A$, the {\em spectrum} of $\Gamma$ is defined to be the spectrum of $A$. While the structure of a graph determines its spectrum, an important converse question in spectral graph theory is to what extent the spectrum determines the structure of a graph; see for example the classical textbooks of Biggs \cite{biggs}, and Cvetkovi\'c, Doob, and Sachs \cite{cds95}. In particular, a persisting problem is to show whether or not a given graph is completely determined by its spectrum (see van Dam and Haemers~\cite{vdh09}), which has led to the following.

\begin{guess}[\cite{vdh09}]
	Almost every graph is determined by its spectrum.
\end{guess}

\noindent The table below, showing the ratio of graphs on $n\le 12$ vertices not determined by their spectrum (`n-DS') to all graphs of order $n$, may be taken as a support of this conjecture.

\begin{center}
\begin{tabular}{|c@{~~}||c@{~~}c@{~~}c@{~~}c@{~~}c@{~~}c@{~~}c@{~~}c@{~~}c@{~~}c@{~~}c@{~~}c@{~~}|}
	\hline
	$n$ & 1 & 2 & 3 & 4 & 5 & 6 & 7 & 8 & 9 & 10 & 11 & 12 \\
	\hline
n-DS & 0 & 0 & 0 & 0 & 0.059 & 0.064& 0.105 & 0.139 & 0.186 & 0.213 & 0.211 & 0.188\\
\hline
\end{tabular}
\end{center}

Considerable effort has thus been devoted to (partial or entire) identification of spectra of some interesting families of graphs. For instance, Godsil and Hensel \cite{gh92} explicitly studied the problem of determining the spectrum of certain voltage graphs. Lov\'asz \cite{l75} provided a formula that expresses the eigenvalues of a graph admitting a transitive group of automorphisms in terms of group characters. In the particular case of Cayley graphs (when the automorphism group contains a subgroup acting regularly on vertices), Babai \cite{ba79} derived a more handy formula by different methods (and, as a corollary, obtained the existence of arbitrarily many Cayley graphs with the same spectrum). In fact, Babai's formula also applies to digraphs and arc-colored Cayley graphs. Following these findings, Dalf\'o, Fiol, and \v{S}ir\'a\v{n} \cite{dfs19} considered a more general construction and derived a method for determining the spectrum of a regular lift of a `base' (di)graph equipped with an ordinary voltage assignment, or, equivalently, the spectrum of a regular cover of a (di)graph. Recall, however, that by far not all coverings are regular; a description of arbitrary graph coverings by the so-called permutation voltage assignments was given by Gross and Tucker \cite{gt77}.
\smallskip

In this paper, we generalize our previous results to arbitrary lifts of graphs (regular or not). Our method not only gives the complete spectra of lifts but also provides bases of the corresponding eigenspaces, both in a very explicit way.
\smallskip

Of course, as a consequence, our method also furnishes the associated characteristic polynomials. To summarize the work previously done in this direction, Kwak and Lee \cite{kl92} and Kwak and Kwon \cite{kk01} dealt with the cases of characteristic polynomials of lifts with voltages in Abelian or dihedral groups. Furthermore, Mizuno and Sato \cite{ms95} obtained a formula for the characteristic polynomial of a regular covering with voltages in the symmetric group, whereas Feng, Kwak, and Lee \cite{fkl04} also solved the case of graph coverings that are not necessarily regular. In these papers, the characteristic polynomials are essentially given in terms of `large' matrices (based on the results of Kwak and Lee \cite{kl92}, where the adjacency matrix of the lift graph was expressed in terms of Kronecker products).
\smallskip

Our approach, in contrast, uses relatively `small' quotient-like matrices derived from the base graph to completely determine the spectrum of the lift (and eigenspace bases).
\smallskip

The paper is organized as follows. In the next section, we recall the notions of permutation voltage assignments on a graph together with the associated lifts, along with an equivalent description of these in terms of relative voltage assignments, and pointing out also their relationship with regular coverings. In Section \ref{sec:basic},  we revisit some fundamental facts in representation theory of groups, and also present a new result on group  representations, which could be of interest on its own. Section \ref{sec:main-result} deals with the main result, where the complete spectrum of a relative lift graph is determined, together with bases of the associated eigenspaces. Our method is illustrated by an example in Section \ref{sec:example}. Finally, in the last section, we discuss the case of regular lifts of digraphs in the light of applications of group characters.



\section{Lifts and permutation voltage assignments}

Let $\Gamma$ be an undirected graph (possibly with loops and multiple edges) and let $n$ be a positive integer. As usual in algebraic and topological graph theory, we will think of every undirected edge joining vertices $u$ and $v$ (not excluding the case when $u=v$) as consisting of a pair of oppositely directed arcs; if one of them is denoted $a$, then the opposite one will be denoted $a^-$. Let $V=V(\Gamma)$ and $X=X(\Gamma)$ be the sets of vertices and arcs of $\Gamma$. 

\subsection{Permutation voltage assignments}
Let $G$ be a subgroup of the symmetric group ${\rm Sym}(n)$, that is, a permutation group on the set $[n]=\{1,2, \ldots, n\}$. A {\em permutation voltage assignment} on the graph $\Gamma$ is a mapping $\alpha: X\to G$ with the property that $\alpha(a^-)=(\alpha(a))^{-1}$ for every arc $a\in X$. Thus, a permutation voltage assignment allocates a permutation of $n$ letters to each arc of the graph, in such a way that a pair of mutually reverse arcs forming an undirected edge receives mutually inverse permutations.

The graph $\Gamma$ together with the permutation voltage assignment $\alpha$ determine a new graph $\Gamma^{\alpha}$, called the {\em lift} of $\Gamma$, which is defined as follows. The vertex set $V^{\alpha}$ and the arc set $X^{\alpha}$ of the lift are simply the Cartesian products $V\times [n]$ and $X\times [n]$. As regards incidence in the lift, for every arc $a\in X$ from a vertex $u$ to a vertex $v$ for $u,v\in V$ (possibly, $u=v$) in $\Gamma$ and for every $i\in [n]$, there is an arc $(a,i)\in X^{\alpha}$ from the vertex $(u,i)\in V^{\alpha}$ to the vertex $(v,i\alpha(a))\in V^{\alpha}$. Note that we write the argument of a permutation to the left of the symbol of the permutation,
with the composition being read from the left to the right.%
\smallskip

If $a$ and $a^{-}$ are a pair of mutually reverse arcs forming an undirected edge of $G$, then for every $i\in [n]$ the pair $(a,i)$ and $(a^-,i\alpha(a))$ form an undirected edge of the lift $\Gamma^{\alpha}$, making the lift an undirected graph in a natural way.
If $\Gamma$ is connected, then the lift $\Gamma^{\alpha}$ is connected if and only if the {\em local group} at a vertex $u\in V$ (that is, the subgroup of $G$ generated by the voltage products on all closed walks rooted at $u$) is a {\em transitive} permutation group on the set $[n]$, see Ezell \cite{e79}.
\smallskip

\subsection{Coverings and permutation voltage assignments}

The mapping $\pi:\Gamma^{\alpha}\to \Gamma$ that is defined by erasing the second coordinate, that is, $\pi(u,i)=u$ and $\pi(a,i)=a$, for every $u\in V$, $a\in X$ and $i\in [n]$, is a {\em covering}, in its usual meaning in algebraic topology. Since $\pi$ consists of $n$-to-$1$ mappings $V^{\alpha}\to V$ and $X^{\alpha}\to X$, we speak about an {\em $n$-fold covering}; we note that this covering may not be regular, see, for instance, Gross and Tucker \cite{gt77}. The graph $\Gamma$ is often called the {\em base graph} of the covering.

Conversely, given an arbitrary $n$-fold covering $\vartheta$ (regular or not) of our base graph $\Gamma$ by some graph $\tilde\Gamma$ with vertex set $\tilde V$ and arc set $\tilde X$, then $\vartheta$ is equivalent to a covering described as above. In more detail, being a covering means that $\vartheta:\tilde\Gamma \to\Gamma$ is assumed to induce $n$-fold mappings $\tilde V\to V$ and $\tilde X\to X$ with the property that for every vertex $u\in V$ and for every vertex $\tilde u\in \vartheta^{-1}(u)\subset \tilde V$ the set of arcs of $\tilde\Gamma$  emanating from $\tilde u$ are mapped bijectively by the {\em same} mapping $\vartheta$ onto the set of arcs of $\Gamma$ emanating from $u$. Then, there is a permutation voltage assignment $\alpha: X\to {\rm Sym}(n)$ inducing a covering $\pi: \Gamma^{\alpha}\to \Gamma$ as above, and a graph isomorphism $\omega: \tilde\Gamma\to \Gamma^{\alpha}$, such that $\vartheta=\omega\pi$, as shown in the following commutative diagram.
\smallskip
\smallskip
\[
\Large{
\xymatrix{
\Gamma^{\alpha}\ar[d]_{\pi} & \ar[l]_{\omega} \ar[ld]^{\vartheta}     \tilde{\Gamma}  \\
\Gamma  &
} }
\]

\subsection{Relative voltage assignments}

Permutation voltage assignments and the corresponding lifts and covers can equivalently be described in the language of the so-called relative voltage assignments as follows. Let $\Gamma$ be the graph considered above, $K$ a group and $L$ a subgroup of $K$ of index $n$; we let $K/L$ denote the set of right cosets of $L$ in $K$. Furthermore, let $\beta: X\to K$ be a mapping satisfying $\beta(a^-)= (\beta(a))^{-1}$ for every arc $a\in X$; in this context, one calls $\beta$ a {\em voltage assignment in $K$ relative to $L$}, or simply a relative voltage assignment. The {\em relative lift} $\Gamma^{\beta}$ has vertex set $V^{\beta}= V\times K/L$ and arc set $X^{\beta}=X\times K/L$. Incidence in the lift is given as expected: If $a$ is an arc from a vertex $u$ to a vertex $v$ in $\Gamma$, then for every right coset $J\in K/L$ there is an arc $(a,J)$ from the vertex $(u,J)$ to the vertex $(v,J\beta(a))$ in $\Gamma^{\beta}$. We slightly abuse the notation and denote by the same symbol $\pi$ the covering $\Gamma^{\beta}\to \Gamma$ given by suppressing the second coordinate.
\smallskip


There is a 1-to-1 correspondence between permutation and relative voltage assignments on a connected base graph. Namely, if $\alpha$ is a permutation voltage assignment as in the first paragraph,
then the corresponding relative voltage group is $K=G$, with a subgroup $L$ being the stabilizer in $K$ of an arbitrary point from the set $[n]$. The corresponding relative assignment $\beta$ is simply identical to $\alpha$, but acting by right multiplication on $K/L$. Conversely, if $\beta$ is a voltage assignment in a group $K$ relative to a subgroup $L$ of index $n$ in $K$, then one may identify the set $K/L$ with $[n]$ in an arbitrary way, and then $\alpha(a)$ for $a\in X$ is the permutation of $[n]$ induced (in this identification) by right multiplication on the set of (right) cosets $K/L$ by $\beta(a)\in K$.
\smallskip

We also note that a covering $\Gamma^{\beta}\to \Gamma$, described in terms of a permutation voltage assignment, is regular if and only if the group $G = \langle \alpha(a):a\in X\rangle$
is a regular permutation group on $[n]$. If the covering is given by a voltage assignment in a group $K$ relative to a subgroup $L$, then it is equivalent to a regular covering if and only if $L$ is a normal subgroup of $K$. In such a case, the covering admits a description in terms of ordinary voltage assignment in the factor group $K/L$ and with voltage $L\beta(a)$ assigned to an arc $a\in X$ with original relative voltage $\beta(a)$. For more details about lifts and coverings, see Gross and Tucker \cite{gt87}.
\smallskip

At this point, it may be instructive to point out a relationship between ordinary and relative voltage assignments and the corresponding lifts. If $\beta$ is a voltage assignment on our graph $\Gamma$ in a group $K$ relative to a subgroup $L$ of index $n$, then the same $\beta$ can also be considered to be an {\em ordinary} voltage assignment in the group $K$ (with no reference to $L$ or, equivalently, letting $L$ be the trivial group). The {\em ordinary lift} of $\Gamma$, which we distinguish by the notation $\Gamma_0^\beta$, has vertex set $V\times K$ and arc set $X\times K$, and for every arc $a\in X$ from $u$ to $v$ and for every $g\in K$ there is an arc $(a,g)$ from the vertex $(u,g)$ to the vertex $(v,g\alpha(a))$ in $\Gamma_0^\beta$. Observe that this gives rise to two coverings: The covering $\Gamma_0^{\beta}\to \Gamma^\beta$ given on vertices and arcs of $\Gamma_0^{\beta}$ by $(u,g)\mapsto (u,Lg)$ and $(a,g)\mapsto (e,Lg)$, followed by the covering $\pi: \Gamma^\beta \to\Gamma$ considered before and given by erasing the second coordinate. Their composition is a regular covering $\Gamma_0^\beta\to\Gamma$ given, again, by suppressing the group coordinate.
\smallskip

\subsection{An example}
The following example is drawn from Gross and Tucker \cite{gt77}.

\begin{example} 
\label{E:1}
{\rm Let $\Gamma$ be the dumbbell graph with vertex set $V=\{u,v\}$ and arc set $X=\{a,a^-,b,b^-, c,c^-\}$, where the pairs $\{a,a^-\}$, $\{b,b^-\}$ and $\{c,c^-\}$ correspond, respectively, to a loop at $u$, an edge joining $u$ to $v$, and a loop at $v$. Let $n=3$ and let $\alpha$ be a voltage assignment on $\Gamma$ with values in ${\rm Sym}(3)$ given by $\alpha(a)=(23)$, $\alpha(b)=e$ (the identity element), and $\alpha(c)=(12)$, so that the voltage group $G$ is equal to $\langle (12),(23)\rangle = {\rm Sym}(3)$.
\smallskip

\begin{figure}[t]
	\begin{center}
		\scalebox{0.68}
		{\includegraphics{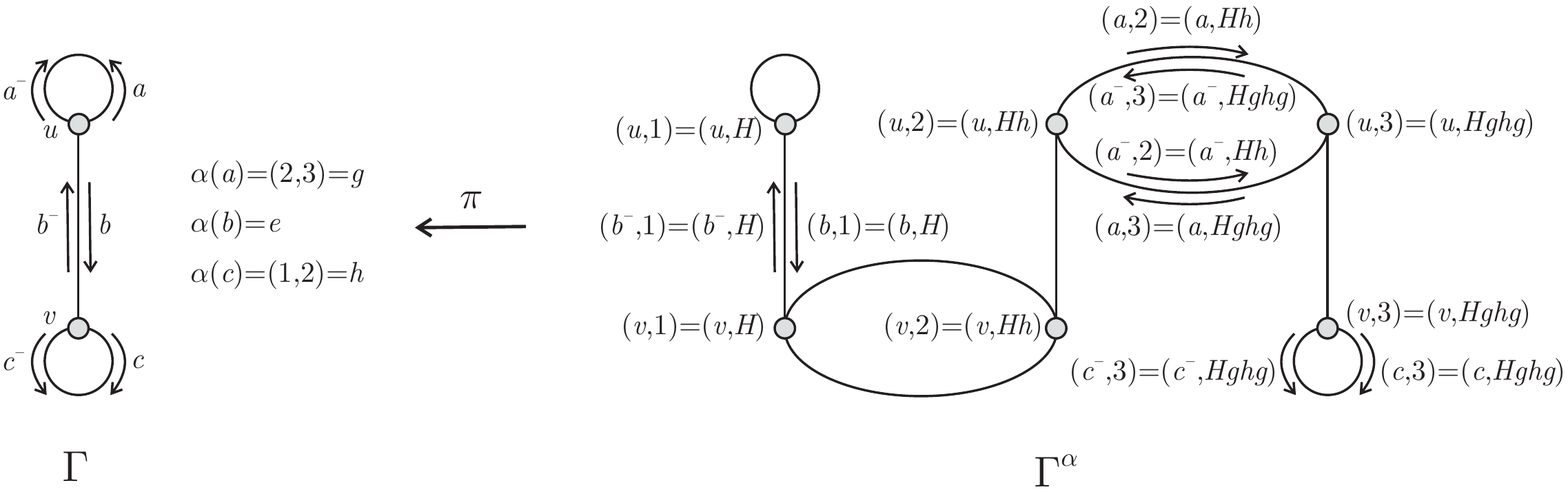}}
		\vskip -9cm
		\caption{The dumbbell graph $\Gamma$ and its relative lift $\Gamma^{\alpha}$. }
		\label{fig1}
	\end{center}
\end{figure}

We may describe the situation in terms of a relative voltage assignment. To do so, we let $K$ be equal to $G$, and for $L$ we take the subgroup $H={\rm Stab}_G(1)= \{e,(23)\}$ of $G$. Furthermore, let the set $[3]=\{1,2,3\}$ be identified with $G/H$ by $1\mapsto H$, $2\mapsto H(12)$, and $3\mapsto H(13)$, and let $\alpha$ be as above but acting by a right multiplication on the right cosets of $H$ in $G$. Then, for example, in the lift $\Gamma^{\alpha}$ the arc $(a,2)$ emanates from the vertex $(u,2)$ and terminates at the vertex $(u,2\alpha(a))=(u,3)$ in the permutation voltage setting. Equivalently, in the relative voltage setting, with $2$ and $3$ identified with $H(12)$ and $H(13)$, the corresponding arc $(a,H(12))$ starting at the vertex $(u,H(12))$ points at the vertex $(u,H(12)\alpha(a)) =(u,H(13))$ because $H(12)\alpha(a) = H(12)(23) = \{(13),(132)\} = H(13)$. The covering $\pi: \Gamma^{\alpha}\to \Gamma$ is irregular since $H$ is not a normal subgroup of $G$; the situation is displayed in Figure 1.}
\end{example}

\section{Some results on group representations}\label{sec:basic}

In this section, we recall some basic results on representation theory that will be used in our study; see, for instance, the textbooks of Burrow~\cite{b93}, and James and Liebeck~\cite{JaLi}. With the same purpose, we also give a new result on group representations, which may be of interest on its own.
\smallskip

We begin by stating a result well known as the Great Orthogonality Theorem. For a complex representation $\rho$ of a group $G$, we let $d_\rho=\dim(\rho)$ denote the dimension of $\rho$; as usual, $\rho(g)_{i,j}$ will be the $(i,j)$-th entry of the matrix $\rho(g)$ for $g\in G$. We denote by $z^*$ the complex conjugate of a complex number $z$. Furthermore, let $\Irep(G)$ denote a complete set of unitary irreducible representations of $G$. Let us also recall the Kronecker symbol $\delta_{s,t}$, for $s,t$ in any suitable domain, equal to $0$ if $s\ne t$ and to $1$ if $s=t$.

\begin{theorem}[The Great Orthogonality Theorem]
Let $G$ be a finite group. Then, for any $\rho,\rho'\in \Irep(G)$,
\[  \sum_{g\in G} \rho(g)_{i,j}\rho'(g)^*_{i',j'} = \frac{|G|}{d_\rho}\delta_{i,i'}\delta_{j,j'}\delta_{\rho, \rho'} \ .\]
\end{theorem}

The second result we will need later is also well known.

\begin{theorem}[Orthogonality Theorem for Characters]
	\label{th:ortocharac}
	Let $G$ be a finite group and for $g\in G$ let $c_g$ denote the size of the conjugacy class of $g$ in the group $G$. Furthermore, let $\chi^\rho$ be the character of an irreducible representation $\rho\in\Irep(G)$.
\begin{itemize}
\item[$(i)$]
For any $\rho,\rho'\in \Irep(G)$,
		$$
		\sum_{g\in G}\chi^\rho(g)\chi^{\rho'}(g)^*=|G|\delta_{\rho,\rho'}.
		$$
\item[$(ii)$]
For any $g,g'\in G$,
		$$
		\sum_{\rho\, \in\, \Irep(G)}\chi^\rho(g)\chi^\rho(g')^*=\frac{|G|}{c_{g}}\delta_{g,g'}
		$$
		where $\delta_{g,g'}=1$ if $g$ and $g'$ belong to the same conjugacy class of $G$, and $\delta_{g,g'}=0$ otherwise.
	\end{itemize}
\end{theorem}

In the proof of our main result in Section \ref{sec:main-result}, we will also need the following proposition, which we could not locate in the literature and which could be of interest on its own. For a subgroup $H$ of a finite group $G$ and for any $\rho\in \Irep(G)$, we let $\rho(H)=\sum_{h\in H}\rho(h)$; that is, $\rho(H)$ is the sum of $d_\rho$-dimensional complex matrices $\rho(h)$ taken over all elements $h$ of the subgroup $H$. We note that, in general, it may happen that $\rho(H)$ is the zero matrix, although, of course, all the matrices $\rho(h)$ for $h\in H$ are non-singular. As usual, the symbol $\rank(M)$ stands for the rank of a matrix $M$.

\begin{proposition}\label{p:id}
For every group $G$ and every subgroup $H$ of $G$ of index $n$,
\begin{equation}\label{eq:id}
\sum_{\rho\, \in\, \Irep(G)} \dim(\rho)\cdot \rank(\rho(H)) = n\ .
\end{equation}
\end{proposition}

\begin{proof}
We will freely use a few facts about irreducible representations and characters of groups, taken from James and Liebeck~\cite{JaLi}. Let $\rho|H$ be the representation of the subgroup $H$ of $G$ induced by $\rho$. If $\sigma$ is any non-trivial irreducible constituent of the decomposition of $\rho|H$ into an irreducible representation of $H$, then, by the elementary representation theory (a particular case of the Great Orthogonality Theorem), $\sum_{h\, \in\,  H} \sigma(h)=0$. On the other hand, for a trivial representation $\tau$ of $H$, we have $\sum_{h\, \in\,  H}\tau(h)=|H|\ne 0$. It follows that $\rank(\sum_{h\, \in\, H}\rho(h)) = s$ if and only if the trivial representation appears exactly $s$ times in the decomposition of $\rho|H$ into irreducible constituents.
Therefore, if the character $\chi^{\rho|H}$ decomposes into irreducible characters of $H$ in the form
\begin{equation*}
\chi^{\rho|H} = \sum_{\sigma\, \in\, \Irep(H)} a_{\sigma} \chi^{\sigma}\ ,
\end{equation*}
then, by the above argument, $s=a_\iota$, where $\iota=\iota_H$ denotes here the trivial representation of $H$. Moreover, by the first form of the Orthogonality Theorem for Characters,
one has
\begin{equation}
\label{eq:10} s = a_\iota = \frac{1}{|H|} \sum_{h\, \in\, H} \chi^{\rho|H}(h)\chi^\iota(h)^* = \frac{1}{|H|} \sum_{h\, \in\, H} \chi^{\rho|H}(h) = \langle \chi^{\rho|H}, \chi^{\iota} \rangle,
\end{equation}
where $\langle \cdot,\cdot\rangle$ denotes the standard inner product of characters relative to $H$. We now invoke the last identity of the proof of Proposition 20.4 on page 213 of \cite{JaLi}, by which
\begin{equation}
\label{eq:20}
\sum_{\rho\, \in\, \Irep(G)} \chi^\rho(e) \langle \chi^{\rho|H},\psi \rangle = \frac{|G|}{|H|}\psi(e)
\end{equation}
for any character $\psi$ of $G$; here $e$ is the unit element of $G$. Letting $\psi$ be the trivial character, using $n=|G|/|H|$, and realizing that $\chi^\rho(e)=\dim(\rho)=d_\rho$, the equation (\ref{eq:20}) reduces to
\begin{equation}
\label{eq:30}
\sum_{\rho\, \in\, \Irep(G)} d_\rho\cdot \langle \chi^{\rho|H},\chi^{\iota} \rangle = n.
\end{equation}
Finally, by (\ref{eq:10}) we know that $\rank(\rho(H))= \rank(\sum_{h\, \in\, H}\rho(h)) = s=\langle \chi^{\rho|H},\chi^{\iota} \rangle$, which in combination with (\ref{eq:30}) gives the equation (\ref{eq:id}), as expected.
\end{proof}

Note that Proposition \ref{p:id} generalizes the well-known identity $\sum_{\rho\, \in\, \Irep(G)} (d_\rho)^2=|G|$ by letting $H=1$, with $\sum_{h\in H}\rho(h)$ reducing to the identity matrix of rank equal to $d_\rho=\dim(\rho)$. In the proof, this shows up by the fact that $\chi^{\rho|H}$ is equal to $d_\rho$ if $H$ is trivial.

\section{The spectrum of a relative lift}
\label{sec:main-result}

Let $\Gamma$ be a connected graph on $k$ vertices (with loops and multiple edges allowed), and let $\alpha$
be a permutation voltage assignment on the arc set $X$ of $\Gamma$ in a {\em transitive} permutation group $G$
of degree $n$. Letting, without loss of generality, $H={\rm Stab}_G(1)$, we will freely use the fact that
$\alpha$ is equivalent to the assignment in $G$ relative to $H$ given by the same mapping $\alpha$,
but it is considered to act on right cosets of $G/H$ by right multiplication.
\smallskip

To the pair $(\Gamma,\alpha)$ as above, we assign the $k\times k$ {\em base matrix} $B$, a square matrix whose rows and columns are indexed by elements of the vertex set of $\Gamma$, and whose $uv$-th element $B_{uv}$ is determined as follows: If $a_1,\ldots,a_j$ is the set of all the arcs of $\Gamma$ emanating from $u$ and terminating at $v$ (not excluding the case $u=v$), then
\begin{equation}\label{eq:Balpha} B_{u,v}=\alpha(a_1)+\cdots + \alpha(a_j)\ ,\end{equation}
the sum being an element of the complex group algebra $\C(G)$; otherwise, we let $B_{u,v}=0$.
\smallskip

As before, let $\rho \in \Irep(G)$ be a unitary irreducible representation of $G$ of dimension $d_\rho$. For our graph $\Gamma$ on $k$ vertices, the assignment $\alpha$ in $G$ relative to $H$, and the base matrix $B$, we let $\rho(B)$ be the $d_\rho k\times d_\rho k$ matrix obtained from $B$ by replacing every non-zero entry $B_{u,v} = \alpha(a_1) +\cdots + \alpha(a_j) \in \C(G)$ as in \eqref{eq:Balpha} by the $d_\rho\times d_\rho$ matrix $\rho(B_{u,v})$ defined by
\begin{equation}\label{eq:B}
\rho(B_{u,v}) =  \rho(\alpha(a_1)) + \cdots + \rho(\alpha(a_j)),
\end{equation}
and by replacing zero entries of $B$ by all-zero $d_\rho\times d_\rho$ matrices. We will refer to $\rho(B)$ as the {\em $\rho$-image} of the base matrix $B$.
\smallskip

An important role in our consideration will be played by the matrix $\rho(H)= \sum_{h\in H}\rho(h)$, that   we have encountered in Proposition \ref{p:id}; let $r_{\rho,H}=\rank(\rho(H))$. With this notation, we return to the $\rho$-image of the base matrix, and we let $\Sp(\rho(B))$ denote the spectrum of $\rho(B)$, that is, the multiset of all the $d_\rho k$ eigenvalues of the matrix $\rho(B)$. Finally, let $r_{\rho,H}\cdot \Sp(\rho(B))$ be the multiset of $d_\rho kr_{\rho,H}$ values obtained by taking each of the $d_\rho k$ entries of the spectrum $\Sp(\rho(B))$ exactly $r_{\rho,H}=\rank(\rho(H))$ times. This, in particular, means that if $r_{\rho,H}=\rank(\rho(H))=0$, the set $r_{\rho,H}\cdot \Sp(\rho(B))$ is empty.
\smallskip

With this terminology and notation, we are ready to state and prove our main result.

\begin{theorem}
\label{t:spec}
Let $\Gamma$ be a base graph of order $k$ and let $\alpha$ be a voltage assignment on $\Gamma$ in a group $G$ relative to a subgroup $H$ of index $n$ in $G$. Then, the multiset of the $kn$ eigenvalues of the relative lift $\Gamma^{\alpha}$ is obtained by concatenation of the multisets $\rank(\rho(H))\cdot  \Sp(\rho(B))$, ranging over all the irreducible representations $\rho\in \Irep(G)$.
\end{theorem}

\begin{proof}
Let $A^{\alpha}$ be the adjacency matrix of the relative lift $\Gamma^{\alpha}$ with vertex set $V^{\alpha}=V\times G/H$, where $V$ is the vertex set of the base graph $\Gamma$, and $G/H$ is the set of right cosets of $H$ in $G$. In more detail, $A^{\alpha}$ is a $kn\times kn$ matrix indexed by ordered pairs $(u,J)\in V^{\alpha}$ whose entries are given as follows. If there is no arc from $u$ to $u'$ for $u,u'\in V$ in $\Gamma$, then there is no arc from any vertex $(u,J)$ to any vertex $(u',J')$ for $J,J'\in G/H$ in the lift $\Gamma^{\alpha}$. In the case of existing adjacencies, if for some $g\in G$, there are exactly $t$ arcs $a_1,\ldots,a_t$ from some vertex $u$ to some vertex $v$ in $\Gamma$ that carry the voltage $g=\alpha(a_1)=\cdots =\alpha(a_t)$, then for every right coset $J\in G/H$ the $(u,J)(v,Jg)$-th entry of $A^{\alpha}$ is equal to $t$.

\smallskip
Our aim is to show that the multiset of all the eigenvalues of the matrix $A^{\alpha}$ is equal to a  concatenation of multisets coming from {\em very different} matrices. Namely, those of the form $\rho(B)$, where in addition each entry in their spectrum is taken $\rank(\rho(H))$ times, and with $\rho$ ranging over all the irreducible representations of $G$. To do so, we proceed as follows.

\smallskip
Take a representation $\rho\in \Irep(G)$ of dimension $d=d_{\rho}$ and consider the matrix $\rho(B)$, the $\rho$-image of the base matrix $B$ of $\Gamma$ relative to $H$. Let $D=D_\rho$ be a diagonal matrix formed by the multiset $\Sp(\rho(B))$ of $dk$ eigenvalues of $\rho(B)$. Furthermore, let $U=U_\rho$ be a $dk\times dk$ matrix whose columns are eigenvectors of $\rho(B)$ corresponding to the eigenvalues in the main diagonal of $D$, so that the three $dk\times dk$ matrices $\rho(B)$, $D$ and $U$ satisfy
\begin{equation}
\label{eq:UD} \rho(B) U = U D.
\end{equation}

Assume that the base matrix $B$ is indexed by the elements of the vertex set $V$ of $\Gamma$ in some fixed order, and that this indexing is lifted to the three matrices above, so that, say, for $u\in V$ the matrix $U$ is formed by $k$ strips of width $d$, indexed by $(u,i)$ for a particular $u\in V$, where $i$ ranges over the set $[d]=\{1,2,\ldots,d\}$. For any $u\in V$, let ${\bf x}_u$ be such a $d\times dk$ strip of $U$. The product ${\bf x}_uD$ is a $d\times dk$ strip taken from the right-hand side of (\ref{eq:UD}), consisting of a `row' of $d\times d$ square matrices of the form ${\bf x}_{u,w}D_w$, where, for each fixed vertex $w\in V$, the square block ${\bf x}_{u,w}$ corresponds to columns indexed by $(w,j)$ for $j\in [d]$, and $D_w$ is the square diagonal block of $D$ corresponding to $w$; one should keep in mind that these blocks depend on $\rho$ that we omit from our notation. \smallskip

Because of the equation (\ref{eq:UD}) the same square block ${\bf x}_{u,w}$ arises from the left-hand side of (\ref{eq:UD}) as the product of the horizontal $d\times dk$ strip of $\rho(B)$ corresponding to the vertex $u$ with the vertical $dk\times d$ strip $U_w$ of $U$ corresponding to the vertex $w$. By (\ref{eq:B}), the horizontal $d\times dk$ strip of $\rho(B)$ consists of square $d\times d$ blocks of the form $\rho(B_{u,v})$ with $v$ ranging over the set $V$. The vertical $dk\times d$ strip $U_w$ corresponding to $w$ consists of square $d\times d$ blocks ${\bf x}_{v,w}$ mentioned above, again with $v$ ranging over $V$. Using the symbol $v\sim u$ to capture {\em all} arcs from $u$ to $v$, it follows that the product of these two strips is a $d\times d$ square block of the form $\sum_{v\sim u} \rho(B_{u,v}){\bf x}_{v,w}$, and by the previous paragraph this square block is also equal to ${\bf x}_{u,w}D_w$. Since $w\in V$ was arbitrary, we conclude that, for every $u\in V$,
\begin{equation}
\label{eq:xD}
\sum_{v\, \sim\, u} \rho(B_{u,v}){\bf x}_{v} = {\bf x}_uD. \end{equation}
Next, for every vertex $(u,Hg)$ of the relative lift $\Gamma^{\alpha}$ consider the $d\times dk$ matrix \begin{equation}\label{eq:Phi} \Phi_{(u,Hg)} = \left(\sum_{z\, \in Hg}\rho(z)\right){\bf x}_u = \left(\sum_{h\, \in\, H}\rho(h)\right)\rho(g){\bf x}_u = \rho(H)\rho(g){\bf x}_u.\end{equation}
A straightforward calculation using (\ref{eq:xD}) gives
\[
\Phi_{(u,Hg)}D = \rho(H)\rho(g){\bf x}_uD = \rho(H)\rho(g)\sum_{v\, \sim\, u}\rho(B_{u,v}){\bf x}_{v}
= \sum_{v\, \sim\, u} \rho(H)\rho(g\alpha(u,v)){\bf x}_v,
\]
where we emphasize again that the summation symbol $v\sim u$ means summing, for every $v\in V$, over {\em all} arcs from $u$ to $v$, which is in accordance with $B_{u,v}$ expressed earlier as in (\ref{eq:Balpha}). But the last term in the chain of the equations above is, by (\ref{eq:Phi}), equal to $\sum_{v\sim u} \Phi_{(v,Hg\alpha(u,v))}$, so that we have derived \begin{equation}
\label{eq:PhiD}
\Phi_{(u,Hg)}D = \sum_{v\sim u} \Phi_{(v,Hg\alpha(u,v))}.
\end{equation}

To get insight into the equation (\ref{eq:PhiD}), recall that $D$ consists of square diagonal blocks $D_w$ for each vertex $w\in V$; let $\mu_{(w,i)}$ for $i\in [d]$ be the eigenvalues of the matrix $\rho(B)$ appearing as diagonal entries of $D_w$. (We point out again that the eigenvalues depend on the representation $\rho$ we have been working with, so that one should write $D_w=D_{\rho,w}$ and $\mu_{(w,i)}=\mu_{(w,i)}(\rho)$, but we chose to keep things simple if no confusion is likely.) For any $v\in V$, let $\Phi_{(v,Hg),w}$ be the square $d\times d$ block of $\Phi_{(v,Hg)}$ corresponding to a vertex $w\in V$. The equation (\ref{eq:PhiD}) now states that, for any fixed $i\in [d]$ and any fixed vertex $w\in V$, the $i$-th column of the square block $\Phi_{(u,Hg),w}$ multiplied by $\mu_{(w,i)}$ is equal to the sum of the $i$-th columns of the square blocks $\Phi_{(v,Hg\alpha(u,v)),w}$, the summation ranging over all $v\sim u$.
\smallskip

The conclusion of the last paragraph means that for every pair of fixed elements $j\in [d]$ and $(w,i)\in V\times [d]$, the $kn$-dimensional vector ${\bf y}^H={\bf y}^H_{j,(w,i)}(\rho)$, obtained by taking the $(j,(w,i))$-th entry of every matrix $\Phi_{(v,Hh)}$ for the $kn$ vertices $(v,Hh)$ of $\Gamma^{\alpha}$, has the property that for every $u\in V$ the $\mu_{(w,i)}$-multiple of the $(u,Hg)$-th coordinate of ${\bf y}^H$ is equal to the sum of $(v,Hg\alpha (u,v))$-th coordinates of ${\bf y}$ over all the adjacencies $v\sim u$. In other words, ${\bf y}^H$ can be identified with a function on the vertex set $V^{\alpha}$ of the relative lift $\Gamma^{\alpha}$ with the property that, for any vertex $(u,J)\in V^{\alpha}$, the sum of the values of ${\bf y}^H$ taken over all the adjacencies $v\sim u$ is the $\mu_{(w,i)}$-multiple of the value of ${\bf y}^H$ at $(u,J)$. But, in the case of a non-zero vector, this is the defining property of an eigenvector of the adjacency matrix $A^{\alpha}$ of our lift. Thus, if ${\bf y}^H={\bf y}^H_{j,(w,i)}(\rho)\ne {\bf 0}$, then it is an eigenvector of the relative lift $\Gamma^{\alpha}$ corresponding to the eigenvalue $\mu_{(w,i)}=\mu_{(w,i)} (\rho)$ in the spectrum of the matrix $\rho(B)$.
\smallskip

The question now is if, in this way, one can determine bases for all eigenspaces of the relative lift. Their expected sum of dimensions is $kn$, the number of vertices of $\Gamma^{\alpha}$. As on both sides of (\ref{eq:PhiD}), the matrix $\rho(H)$ appears as a factor, one cannot produce more than $r_{\rho,H}= \rank(\rho(H))$ linearly independent vectors ${\bf y}={\bf y}_{j,(w,i)}(\rho)$ as described above for the eigenvalue $\mu_{(w,i)}(\rho)$. Taking into account that $w$ is one of the $k$ elements of $V$ and invoking Proposition \ref{p:id}, this means that our method furnishes at most \linebreak
 $k\cdot \sum_{\rho\, \in\, \Irep(G)} d_{\rho} r_{\rho,H} = kn$ linearly independent vectors ${\bf y}$ as above. To prove that we obtain a multiset of eigenvalues and bases of eigenspaces with the dimension sum as expected, we only need to show that our procedure gives rise to $kn$ independent eigenvectors of $A^{\alpha}$.
\smallskip

To do so, let us begin with the regular case, which may be identified with the situation when $H$ is the trivial group and $n=|G|$, and let us re-examine (\ref{eq:PhiD}) as follows. Recalling the `sections' ${\bf x}_{u,w}$, the identity (\ref{eq:PhiD}) may be rewritten with the help of (\ref{eq:Phi}) in the form
\begin{equation}\label{eq:PhiDw} \rho(g){\bf x}_{u,w}D_w = \sum_{v\sim u} \rho(g\alpha(u,v)){\bf x}_{v,w} \ \end{equation}
for every $u,w\in V$. And, as argued above, (\ref{eq:PhiDw}) means that for a given $\rho\in \Irep(G)$, fixed row and column indices $j,i\in [d]$ for $d=d_\rho={\rm dim}(\rho)$, and a given $w\in V$, the collection of the $(j,i)$-th entries of the product $\rho(g){\bf x}_{u,w}$ of the two $d\times d$ matrices for varying $g\in G$ and $u\in V$ determine a $k|G|$-dimensional vector ${\bf y} = {\bf y}_{j,(w,i)} = {\bf y}_{j,(w,i)}(\rho)$ that, if non-zero, is an eigenvector for the eigenvalue $\mu_{(w,i)}= \mu_{(w,i)}(\rho)$.
\smallskip

This can be expressed in the following equivalent form. For $j\in [d]$ let $R_{\rho,j}$ be the $|G|\times d$ matrix whose row indexed by $g\in G$ is equal to the $j$-th row of $\rho(g)$. We will be assuming the same indexing of the rows by elements $g\in G$ in {\em all} $R_{\rho,j}$ for $j\in [d]$. The $i$-th column of the product $R_{\rho,j}{\bf x}_{u,w}$ will then account for the $|G|$ coordinates of the vector ${\bf y}$ that corresponds to a fixed $u\in V$ but varying $g\in G$. Furthermore, let $S_{\rho,j} = {\rm diag}(R_{\rho,j},\ldots, R_{\rho,j})$ be the $k|G|\times kd$ block-diagonal matrix consisting of $k$ diagonal blocks, each equal to $R_{\rho,j}$. Here we assume that the row- and column-indexing of the blocks of $S_{\rho,j}$ by elements of the set $V$ is the same as used for the matrix $U=U_\rho$ at the beginning of the proof. Last, recall the matrix $U_w=U_{\rho,w}$, which may also be viewed as a concatenation of $k$ square $d$-dimensional sections ${\bf x}_{u,w}$ into one `column of blocks' for varying $u\in V$. Then, the entire vector ${\bf y}= {\bf y}_{j,(w,i)} (\rho)$ can be identified with the $i$-th column of the product $S_{\rho,j} U_{\rho,w}$. Since the matrix $U = U_\rho$ is just a concatenation of the strips $U_{\rho,w}$ for $w\in V$, extending the last finding over all $w\in V$, we conclude that:
\medskip

\noindent{\bf Claim 1.} {\sl The collection of $kd$ vectors ${\bf y}= {\bf y}_{j,(w,i)}(\rho)$ for varying $w\in V$ and $i\in [d]$ can be identified with the collection of all the $kd$ columns of the product $S_{\rho,j} U_\rho$.}
\medskip

Continuing in this fashion, let $S_\rho$ be the $|G|k\times d^2k$ block matrix of the form $S_\rho=(S_{\rho,1}, \ldots, S_{\rho,j}, \ldots, S_{\rho,d})$ for $d=d_\rho$ obtained by concatenating the matrices $S_{\rho,j}$ for $j\in [d]$. Furthermore, let $T_\rho={\rm diag}(U_\rho,\ldots, U_\rho)$ be the $kd^2\times kd^2$ block-diagonal matrix formed by $d$ identical diagonal blocks equal to $U_\rho$. By the same token as before, Claim 1 may then be extended to stating that the vectors ${\bf y}= {\bf y}_{j,(w,i)}(\rho)$, for varying $w\in V$ and $i\in [d]$ and also for varying $j\in [d]$, can be identified with the $d^2k$ columns of the product $S_\rho T_\rho$.
\smallskip

We are now ready to perform the final step of our procedure to determine all the eigenspaces of the lift. Let $S = (S_{\rho})_{\rho\, \in\, \Irep(G)}$ be the $|G|k\times |G|k$ square matrix obtained by concatenating the $|\Irep(G)|$ blocks $S_\rho$ of dimension $|G|k\times d_\rho k$ for $\rho\in \Irep(G)$ into one `thick row'. Finally, let $T = {\rm diag}(T_\rho)_{\rho\, \in\, \Irep(G)}$ be the block-diagonal $|G|k\times |G|k$ matrix consisting of $d_\rho$ blocks of type $d_\rho k \times d_\rho k$, all equal to $T\rho$; we will assume that both $S$ and $T$ arise from the same ordering of elements $\rho\in \Irep(G)$. Note that the number of columns of $S$ and the dimensions of $T$ are indeed equal to $|G|k$ since $\sum_{\rho\, \in\, \Irep(G)} d_{\rho}^2=|G|$. This allows us to extend Claim 1 and the subsequent statements as follows:
\medskip

\noindent {\bf Claim 2.} {\sl The collection of $k|G| = k\sum_{\rho\, \in\, \Irep(G)} d_{\rho}^2$ vectors ${\bf y}= {\bf y}_{j,(w,i)}(\rho)$ for varying $w\in V$, $i,j\in [d_\rho]$, and $\rho\in \Irep(G)$, can be identified with the collection of all the $k|G|$ columns of the product $ST$. }
\medskip

To prove that the $k|G|$ vectors ${\bf y}= {\bf y}_{j,(w,i)}(\rho)$ for varying $w\in V$, $i,j\in [d]$ and $\rho\in \Irep(G)$ are linearly independent---and hence constituting a complete set of representatives of the eigenspaces of the lift---it is sufficient to show that ${\rm rank}(ST)=k|G|$. Our way of introducing $T$ implies that it is a block-diagonal matrix consisting, for each $\rho\in \Irep(G)$, of $d_\rho$ diagonal $d_\rho k\times d_\rho k$ blocks equal to $U_\rho$. Recall now that $U_\rho$ is a $dk\times dk$ matrix whose columns are $dk$ linearly independent eigenvectors of the matrix $\rho(B)$ encountered at the beginning of the proof. This, in combination with our description of $T$, immediately implies that ${\rm rank}(T)=k|G|$. Also, recall that $S$ is a concatenation of $k|G|\times kd_{\rho}^2$ blocks $S_\rho$ for $\rho\in \Irep(G)$, where $S_\rho=(S_{\rho,1}, \ldots, S_{\rho,d_\rho})$ themselves are concatenations of the matrices $S_{\rho,j}={\rm diag}(R_{\rho,j},\ldots, R_{\rho,j})$ for $j\in [d]$ consisting of $k$ equal diagonal blocks, with $R_{\rho,j}$ being the $|G|\times d$ matrix whose row indexed by $g\in G$ is equal to the $j$-th row of $\rho(g)$. By the indexing convention introduced earlier, the diagonal blocks of $S_{\rho,j}$ can be assumed to be indexed by $(v,v)$ for $v\in V$; such a $v\in V$ will be the {\em position} of the block $R_{\rho,j}$ in $S_{\rho,j}$.
\smallskip

With this preparation, we now show that any two distinct columns of $S$ are orthogonal (obviously, all these columns are non-zero vectors). Indeed, let us consider two distinct columns $c=c(j,i,v,\rho)$ and $c'=c'(j',i',v',\rho')$ of $S$, obtained as extensions of the $i$-th column of $R_{\rho,j}$ at position $v$ and the $i'$-th column of $R_{\rho',j'}$ at position $v'$. If $v\ne v'$, then, by construction of $S$, the dot product $c\cdot c'$ is equal to zero simply because, in the product, every non-zero coordinate of $c$ meets a zero coordinate of $c'$ and vice versa. If $v=v'$, then one just needs to realize that elements of $R_{\rho,j}$ and $R_{\rho',j'}$ are entries $\rho(g)_{j,i}$ and $\rho'(g)_{j',i'}$ of the matrices $\rho(g)$ and $\rho'(g)$ for $g\in G$, respectively. The product $c\cdot c'$ for $v=v'$ thus reduces to $\sum_{g\in G} \rho(g)_{j,i} \rho'(g)_{j',i'}$ that, by the Great Orthogonality Theorem, is equal to zero whenever $j\ne j'$, $i\ne i'$ or $\rho\ne \rho'$. This proves the orthogonality of the columns of $S$.
\smallskip

But then, the rank of $ST$ is equal to $k|G|$, which means that the column vectors of $ST$ give a complete set of representatives of the eigenspaces of the lift $\Gamma^{\alpha}$ in the regular case.
\smallskip

It remains to consider the case when $H$ is an arbitrary subgroup of $G$ of index $n$. Using the previously introduced notation and realizing that $\rho(H)\rho(g)=\rho(Hg)$, in this general case the identity (\ref{eq:PhiD}) with the help of (\ref{eq:Phi}) may be rewritten in the form
\begin{equation}\label{eq:PhiDwH} \rho(J){\bf x}_{u,w}D_w = \sum_{v\sim u} \rho(J\alpha(u,v)){\bf x}_{v,w} \ \end{equation}
for every $u,w\in V$ and each right coset $J\in G/H$. Recall that (\ref {eq:PhiDwH}) now means that for a given $\rho\in \Irep(G)$, fixed row and column indices $j,i\in [d]$ for $d=d_\rho$, and a given $w\in V$, the collection of the $(j,i)$-th entries of the product $\rho(J){\bf x}_{u,w}$ of the two $d\times d$ matrices for varying $J\in G/H$ and $u\in V$ determine a $kn$-dimensional vector ${\bf y}^H = {\bf y}^H_{j,(w,i)}(\rho)$ which, if non-zero, is an eigenvector of the relative lift for the eigenvalue $\mu_{(w,i)}= \mu_{(w,i)}(\rho)$.
\smallskip

Bearing this in mind, it follows that one can go through the same individual steps of the procedure devised for the regular case. The only items in need of modification will be those in the description of the matrix $S$ that refer to row coordinates indexed by elements of $G$; these need to be replaced by {\em row sums} indexed by the corresponding cosets of $G/H$. Thus, instead of $R_{\rho,j}$ we let $R^H_{\rho,j}$ be the $n\times d_\rho$ matrix whose row with index $J\in G/H$ is the $j$-th row of the matrix $\rho(J)=\sum_{g\in J}\rho(g)$. Similarly, $S^H_{\rho,j}={\rm diag}(R^H_{\rho,j})$ will now be the $kn\times kd$ block-diagonal matrix with $k$ blocks each equal to $R^H_{\rho,j}$, and the $kn \times kd^2$ matrix $S^H_\rho$ will be formed by the concatenation of the blocks $S^H_{\rho,j}$ taken over all $j\in [d]$. Finally, the $kn\times k|G|$ matrix $S^H$ will be a concatenation of the blocks $S^H_\rho$ taken over all $\rho\in \Irep(G)$. In all these matrices, we assume the same indexing conventions as before. And, arguing exactly as in the regular case, we obtain:
\medskip

\noindent {\bf Claim 3.} {\sl The collection of $k|G|$ vectors ${\bf y}^H= {\bf y}^H_{j,(w,i)}(\rho)$ of dimension $kn$, for varying $w\in V$, $i,j\in [d_\rho]$ and taken over all $\rho\in \Irep(G)$ can be identified with the collection of all the $k|G|$ columns of the product $S^HT$. }
\medskip

To finish the argument, it remains to show that the rank of $S^HT$ is equal to $kn$. From the previous part, we know that $T$ is non-singular, and so it is sufficient to prove that ${\rm rank}(S^H)=kn$. We derive this from the fact that ${\rm rank}(S)=k|G|$, established earlier for the regular case. The way $S^H$ has been defined implies that $S^H$ is obtained from $S$ by replacing, for every $J\in G/H$ and every fixed $v\in V$, the $|H|$ rows corresponding to the row indices $(g,v)$ for $g\in J$ by the {\em sum} of these $|J|=|H|$ rows. Since we have $k$ possibilities for $v\in V$ and $n$ cosets, and adding all rows corresponding to a coset can be done by $|H|-1$ summations of {\em pairs} of rows, creating this way $S^H$ from $S$ can be done by performing $kn(|H|-1)$ sums of {\em pairs} of rows. By summation of a pair of rows, the rank of the resulting matrix decreases by at most one. Since $n|H|=|G|$, it follows that
\[ {\rm rank}(S^H) \ge {\rm rank}(S) - kn(|H|-1) = k|G| - (k|G|-kn) = kn \]
and, as we trivially have $kn\ge {\rm rank}(S^H)$, we conclude that ${\rm rank}(S^H) = kn$, as claimed. It means that the columns of $S^HT$ indeed account for representatives of all the eigenspaces of the permutation lift $\Gamma^{\alpha}$. This completes the proof.
\end{proof}
\smallskip

In the notation introduced in the above proof, we have the following consequence.
\smallskip

\begin{corollary}\label{cor:main}
All the eigenspaces corresponding to the permutation lift $\Gamma^{\alpha}$ are generated by a subset of $kn$ independent columns of the product $S^HT$. \hfill $\Box$
\end{corollary}

\section{An example}\label{sec:example}

Following the method of Theorem \ref{t:spec} and keeping to the notation introduced in its proof, we will now work out the spectrum of the relative lift from Example \ref{E:1}. Referring to Figure 1, recall that we consider the dumbbell graph $\Gamma$ with vertex set $V=\{u,v\}$ and arc set $X=\{a,a^-,b,b^-, c,c^-\}$, where the pairs $\{a,a^-\}$, $\{b,b^-\}$ and $\{c,c^-\}$ determined a loop at $u$, an edge joining $u$ to $v$, and a loop at $v$, respectively. The permutation voltage assignment $\alpha$ on $\Gamma$ in the group ${\rm Sym}(3)$ was given by $\alpha(a)=(23)$, $\alpha(b)=e$, and $\alpha(c)=(12)$. An equivalent description is to regard $\alpha$ as a relative voltage assignment, with values of $\alpha$ on arcs acting on the right cosets of $G/H$ for $H={\rm Stab}_G(1)=\{e,(23)\}$ by right multiplication. Letting $g=(23)$ and $h=(12)$, the base matrix $B$ of $G$ with entries in the group algebra $\C(G)$ has the form
\begin{equation}
\label{base-matrix}
B = \left(\begin{array}{cc} \alpha(a)+\alpha(a^-) & \alpha(b) \\ \alpha(b^-) & \alpha(c)+\alpha(c^-) \end{array}\right) = \left(\begin{array}{cc} 2g & e \\ e & 2h \end{array}\right).
\end{equation}

The (transitive) voltage group $G={\rm Sym}(3)=\{e,g,h,r,s,t\}$ with $r=ghg=(13)$, $s=gh=(123)$ and $t=hg=(132)$ has a complete set of irreducible unitary representations $\Irep(G)=\{\iota,\pi,\sigma\}$ corresponding to the symmetries of an equilateral triangle with vertices positioned at the complex points $e^{i\frac{2\pi}{3}}$, $1$, and $e^{i\frac{4\pi}{3}}$, where
\begin{align*}
\iota &: G\to \{1\},\quad \dim(\iota)=1\quad \mbox{(the trivial representation)},\\
\pi &: G\to \{\pm 1\},\quad \dim(\pi)=1\quad \mbox{(the parity permutation representation), and},\\
\sigma &:G\to GL(2,\C),\quad \dim(\sigma)=2, \mbox{ generated by the unitary matrices}
\end{align*}
\begin{equation*}
\label{eq:reprho}
\sigma(g) = \frac{1}{2}\left(\begin{array}{cc} -1 & -\sqrt{3} \\  -\sqrt{3} & 1 \end{array}\right)\qquad {\rm and} \qquad
\sigma(h) = \frac{1}{2}\left(\begin{array}{cc} -1 & \sqrt{3} \\ \sqrt{3} & 1 \end{array}\right),
\end{equation*}
whence we obtain
$$
\sigma(r)=\sigma(ghg)= \left(\begin{array}{cc} 1 & 0\\
0 & -1 \end{array}\right),\quad \sigma(s)=\sigma(gh)= \frac{1}{2}\left(\begin{array}{cc} -1 & -\sqrt{3} \\  \sqrt{3} & -1 \end{array}\right),\quad \mbox{and}
$$
$$
\sigma(t)=\sigma(hg)= \frac{1}{2}\left(\begin{array}{cc} -1 & \sqrt{3} \\  -\sqrt{3} & -1 \end{array}\right).
$$
Then, the images of $B$ under these three representations are given by
\begin{equation}
\label{eq:rhoB}
\iota(B) {=}  \left(\begin{array}{cc} 2 & 1 \\ 1 & 2 \end{array}\right),\
	\pi(B) {=} \left(\begin{array}{cc} -2 & 1 \\ 1 & -2 \end{array}\right),\
\sigma(B) = \left(\begin{array}{cccc}
-1 & -\sqrt{3} & 1 & 0 \\  -\sqrt{3} & 1 & 0 & 1 \\ 1 & 0 & -1 &  \sqrt{3} \\ 0 & 1 & \sqrt{3} & 1
\end{array}\right),
\end{equation}
with spectra $\Sp(\iota(B))=\{1,3\}$, $\Sp(\pi(B))=\{-3,-1\}$, and $\Sp(\sigma(B))=\{\pm\sqrt{3},\pm\sqrt{7}\}$. To determine the `multiplication factors' appearing in front of the spectra in the statement of Theorem \ref{t:spec}, we evaluate $\iota(H)=\iota(e)+\iota(g)=1+1=2$, of rank 1, $\pi(H)=\pi(e) + \pi(g) = 1-1=0$, of rank 0, and
\begin{equation}
\label{eq:rH}
\sigma(H)=\sigma(e) + \sigma(g) = \frac{1}{2}\left(\begin{array}{rr} 1 & -\sqrt{3} \\ -\sqrt{3} & 3\end{array}\right),
\end{equation}
of rank 1. By Theorem \ref{t:spec}, the spectrum of the adjacency matrix $A^{\alpha}$ of the relative lift $\Gamma^{\alpha}$ is obtained by concatenating the sets $1\cdot \{1,3\}$, $0\cdot \{-3,-1\}$, and $1\cdot \{\pm \sqrt{3},\pm \sqrt{7}\}$, giving
$$
\Sp(A^{\alpha}) = \{1,\pm\sqrt{3},\pm \sqrt{7},3\}.
$$

Due to the complexity of determining the corresponding eigenspaces in the second part of the proof of Theorem \ref{t:spec}, we illustrate the details of the process on the same example. Let us again begin by considering the regular case, that is, when $H$ is the trivial group. Then, as explained in the proof, the eigenvectors of the ordinary lift $\Gamma_0^{\alpha}$, shown in Figure \ref{fig2}, are obtained from the matrix product $ST$, where, in the previously introduced notation, $S$ is now a $12\times 12$ matrix with block form
$$
S= (S_\iota\, |\, S_\pi\, |\, S_\sigma) = \left(
\begin{array}{cc|cc|cccc}
S_{\iota,1} &  O & S_{\pi,1} & O & S_{\sigma,1} & O & S_{\sigma,2} & O \\
O & S_{\iota,1} &  O & S_{\pi,1} & O & S_{\sigma,1} & O & S_{\sigma,2}
\end{array}\right),
$$
where
\begin{align*}
S_{\iota,1} &=(\iota(e),\iota(g),\iota(h),\iota(r),\iota(s),\iota(t))^{\top}
=(1,1,1,1,1,1)^{\top},\\
S_{\pi,1} &=(\pi(e),\pi(g),\pi(h),\ldots,\pi(t))^{\top}=(1,-1,-1,-1,1,1)^{\top},\\
S_{\sigma,1} &=(\sigma(e)_1,\sigma(g)_1,\ldots,\sigma(t)_1)^{\top}
=
\left(
\begin{array}{cccccc}
1 & -1/2 & -1/2 & 1 & -1/2 & -1/2\\
0 & -\sqrt{3}/2 & \sqrt{3}/2 & 0 & -\sqrt{3}/2 & \sqrt{3}/2
\end{array}
\right)^{\top},
\\
S_{\sigma,2} &=(\sigma(e)_2,\sigma(g)_2,\ldots,\sigma(t)_2)^{\top}
=
\left(
\begin{array}{cccccc}
0 & -\sqrt{3}/2 & \sqrt{3}/2 & 0 & \sqrt{3}/2 & -\sqrt{3}/2\\
1 & 1/2 & 1/2 & -1 & -1/2 & -1/2
\end{array}
\right)^{\top}.
\end{align*}
Note that, as required in the proof, $S_{\sigma,1}$ and $S_{\sigma,2}$ are formed, respectively, out of the first and second rows of the matrices $\sigma(e)$ and $\sigma(g)$, and so on. The matrix $T$ in this case is a $12\times 12$ matrix with block form $T = {\rm diag}(T_\iota,T_\pi,T_\sigma)$, where $T_\iota = U_\iota$, $T_\pi=U_\pi$, and $T_\sigma = {\rm diag}(U_\sigma,U_\sigma)$, so that
$$
T =\left(
\begin{array}{cccc}
U_\iota &  O & O & O \\
O & U_\pi & O & O \\
O & O & U_\sigma & O \\
O & O & O & U_\sigma
\end{array}\right).
$$
Here one needs to be careful about indexation of rows and columns to align eigenvectors with the corresponding eigenvalues. In accordance with the proof of Theorem \ref{t:spec}, for each $\rho\in\{\iota,\pi,\sigma\}$ of dimension $d_\rho$, the  $d_\rho\times d_\rho$ matrix $U_\rho$ is formed by a choice of the corresponding eigenvectors of $\rho(B)$. To proceed, we choose to list the eigenvalues $\mu_{(w,i)}=\mu_{(w,i)}(\rho)$ for $w\in V=\{u,v\}$, $\rho$ as above, and $i\in [d_{\sigma}]$, in the form $\mu_{(u,1)}(\iota)=3$, $\mu_{(v,1)} (\iota)=1$, $\mu_{(u,1)}(\pi)=-3$, $\mu_{(v,1)}(\pi)=-1$, $\mu_{(u,1)}(\sigma)=\sqrt{3}$, $\mu_{(u,2)}(\sigma) =\sqrt{7}$, $\mu_{(v,1)}(\sigma)=-\sqrt{7}$, and $\mu_{(v,2)}(\sigma)=-\sqrt{3}$, together with a choice of the corresponding eigenvectors as follows:
\begin{equation*}\label{eq:Ui}
U_\iota = \left(
\begin{array}{cc}
1  & -1 \\
1  &  1
\end{array}
\right),\quad
U_\pi = \left(
\begin{array}{cc}
-1  & 1 \\
1  &  1
\end{array}
\right),\quad \mbox{and} \quad
U_\rho = \left(
\begin{array}{cccc}
1  & -\sqrt{3}   &   \sqrt{3}   & -1 \\
-1 & \sqrt{7}+2  &  \sqrt{7}-2  & -1 \\
1  &  \sqrt{3}   &  -\sqrt{3}   & -1 \\
1  & \sqrt{7}+2  &  \sqrt{7}-2  &  1
\end{array}\right),
\end{equation*}
where, from left to right, the columns of $U_\iota$ correspond to eigenvalues $3$ and $1$, the columns of $U_\pi$ to the eigenvalues $-3$ and $-1$, and finally the columns of $U_\sigma$ correspond to the eigenvalues $\sqrt{3}, \sqrt{7},\sqrt{-7}$ and $\sqrt{-3}$, respectively. Then, by the proof of Theorem \ref{t:spec}, the product $ST$ yields all the eigenvectors of the adjacency matrix $A^{\alpha}$ of the regular lift, corresponding to the eigenvalues $\pm 3$, $\pm 1$, $\pm \sqrt{3}$, and $\pm \sqrt{7}$ (the last four with multiplicity $2$); see also Section \ref{sec:chi}.
\smallskip

For the lift relative to $H$, note first that the right cosets of $H=\{e,g\}$ are $Hh=\{h,gh\}$ and $Hghg=\{ghg,hg\}$. These correspond, respectively, to the rows $\{1,2\}$, $\{3,5\}$, and $\{4,6\}$ out of the  first 6 rows of the matrix $S$ (and by the form of $S$ it is clearly sufficient to restrict ourselves to its first 6 rows). Then, the new matrix $S^H$ has blocks
\begin{align*}
S^H_\iota &= (\iota(e)+\iota(g),\iota(h)+\iota(s),\iota(r)+\iota(t))^{\top}=(2,2,2)^{\top},\\
S^H_\pi &=(\pi(e)+\pi(g),\pi(h)+\pi(s),\pi(r)+\pi(t))^{\top}=(0,0,0)^{\top},\\
S^H_{\sigma,1} &=(\sigma(e)_1+\sigma(g)_1,\sigma(h)_1+\sigma(s)_1,\sigma(r)_1+\sigma(t)_1)^{\top}
=
\left(
\begin{array}{ccc}
1/2        & -1 & 1/2 \\
-\sqrt{3}/2 & 0  & \sqrt{3}/2
\end{array}
\right)^{\top},
\\
S^H_{\sigma,2} &=(\sigma(e)_2+\sigma(g)_2,\sigma(h)_2+\sigma(s)_2,\sigma(r)_2+\sigma(t)_2)^{\top}
=
\left(
\begin{array}{ccc}
-\sqrt{3}/2 & \sqrt{3} & -\sqrt{3}/2\\
3/2         & 0        & -3/2
\end{array}
\right)^{\top}.
\end{align*}
As the product $S^HT$, we now obtain the $6\times 12$ matrix whose first four columns (the first two of which  correspond to the eigenvalues $1$ and $-1$) are
{\small \[ \left( \begin{array}{rrrr} 2 & -2 & 0 & 0 \\ 2 & -2 & 0 & 0 \\ 2 & -2 & 0 & 0 \\ 2 & 2 & 0 & 0 \\ 2 & 2 & 0 & 0 \\ 2 & 2 & 0 & 0 \end{array} \right),\]}
and the remaining eight columns of $S^HT$ have the form
{\small $$
\begingroup
\setlength\arraycolsep{2pt}
\left(
\begin{array}{cccccccc}
 \frac12(\sqrt{3}{+}1) & -\frac{\sqrt{3}}{2}(3{+}\sqrt{7}) & \frac{\sqrt{3}}{2}(3{-}\sqrt{7}) &   \frac12(\sqrt{3}{-}1) & -\frac12(3{+}\sqrt{3}) & \frac32(\sqrt{7}{+}3) & \frac32(\sqrt{7}{-}3) & \frac12(\sqrt{3}-3) \\
 -1     & \sqrt{3}  & -\sqrt{3}  & 1      & \sqrt{3}  & -3 &  3 & -\sqrt{3} \\
\frac12(1{-}\sqrt{3}) & \frac{\sqrt{3}}{2}(1{+}\sqrt{7})  & \frac{\sqrt{3}}{2}(\sqrt{7}{-}1) & -\frac12(1{+}\sqrt{3})  & \frac12(3{-}\sqrt{3}) & -\frac32(1{+}\sqrt{7}) & \frac32(1{-}\sqrt{7}) & \frac12(3{+}\sqrt{3}) \\
\frac12(1{-}\sqrt{3}) & -\frac{\sqrt{3}}{2}(1{+}\sqrt{7}) & \frac{\sqrt{3}}{2}(1{-}\sqrt{7}) & -\frac12(1{+}\sqrt{3}) & \frac12(3{-}\sqrt{3}) & \frac32(1{+}\sqrt{7}) & \frac32(\sqrt{7}{-}1)  & \frac12(3{+}\sqrt{3}) \\
 -1     & -\sqrt{3} &  \sqrt{3}  & 1  & \sqrt{3}  & 3  & -3 & -\sqrt{3} \\
\frac12(\sqrt{3}{+}1) & \frac{\sqrt{3}}{2}(3{+}\sqrt{7}) & \frac{\sqrt{3}}{2}(\sqrt{7}{-}3) &  \frac12(\sqrt{3}{-}1) & -\frac12(3{+}\sqrt{3}) & -\frac32(\sqrt{7}{+}3)  & \frac32(3{-}\sqrt{7}) & \frac12(\sqrt{3}-3)
\end{array}
\right),
\endgroup
$$}
with columns from the left to the right corresponding to the eigenvalues $\sqrt{3}$, $\sqrt{7}$, $-\sqrt{7}$, $-\sqrt{3}$, $\sqrt{3}$, $\sqrt{7}$, $-\sqrt{7}$ and $-\sqrt{3}$, as dictated by our chosen indexation of $S$ and $T$. Now, observe that the 9th and the 12th columns of $S^HT$ are a $(-\sqrt{3})$-multiple of the 5th and the 8th columns, respectively; similarly, the 10th and the 11th columns of $S^HT$ are the same $(-\sqrt{3})$-multiple of the 6th and the 7th columns. This means that the eigenspaces of the relative lift $\Gamma^{\alpha}$ for the $6$ eigenvalues $3,1,\sqrt{3},\sqrt{7},-\sqrt{7},-\sqrt{3}$ are all one-dimensional, and may be taken to be generated by the 6-dimensional vectors forming the 1st, 2nd, 5th, 6th, 7th and 8th columns, respectively, of the above matrix $S^HT$. The sum of the dimensions of all the eigenspaces of $\Gamma^{\alpha}$ is $6=2\cdot 3 = kn$, which is in agreement with Corollary \ref{cor:main}.

\section{The regular case with the help of group characters}\label{sec:chi}

If the covering $\Gamma^{\alpha}\to \Gamma$ is regular, the multiset of eigenvalues of the lift can also be obtained using group characters, as it was demonstrated in \cite{dfs19}, even for regular lifts of {\em digraphs}. The eigenvalue multiplicities furnished by \cite{dfs19} are all algebraic---which, for digraphs, may be distinct from geometric multiplicities---and therefore the method of \cite{dfs19} gives no information about eigenspaces in general. To relate the techniques used in the previous sections with those of \cite{dfs19}, we briefly sum up the essentials. This will also offer a complete argument for establishing the key equality \eqref{eq:key} (see below) that appears, in an equivalent form, in the proof of Theorem 2.1 in \cite{dfs19}.
\smallskip

The method of \cite{dfs19} is based on establishing a relationship between counting closed walks in a relative lift of a connected base graph $\Gamma$ (now allowed to contain both undirected and directed edges) and characters of the voltage group $G$. As before, $B$ will be the `voltage adjacency matrix' for $\Gamma$ indexed by its vertex set $V$, with elements $B_{u,v}$ given as in (\ref{eq:Balpha}). An important role in what follows is played by the diagonal elements of the powers $B^\ell$ for $\ell\ge 1$, evaluated on the group algebra $\C(G)$. For $u\in V$, every such diagonal element has the form $(B^\ell)_{u,u} = \sum_{g\in G}a^{(\ell)}_{(u,g)}g \in \C(G)$ for some integers $a^{(\ell)}_{(u,g)}$. Next, for every $\rho\in \Irep(G)$ and for the corresponding character $\chi^\rho$ of $G$ we let
\begin{equation}
	\label{eq:chiB}
	\chi^\rho((B^\ell)_{u,u}) = \sum_{g\in G} a^{(\ell)}_{(u,g)}\chi^\rho(g).
\end{equation}
Note that, in this way, we only extend $\chi^\rho$ over particular linear combinations in $\C(G)$. For the unit element $e\in G$, we will derive a formula for the coefficients $a^{(\ell)}_{(u,e)}$ from (\ref{eq:chiB}) by evaluating the inner product of the vectors ${\bf z}=(\chi^\rho((B^\ell)_{u,u}))_{\rho\, \in\, \Irep(G)}$ and ${\bf z}'=(\chi^\rho(e))_{\rho\, \in\, \Irep(G)}$ in two ways (assuming the same indexing in both vectors). First, using the definition of $\chi^\rho((B^\ell)_{u,u})$ and, in the last step, the second form of the Orthogonality Theorem for Characters from Section \ref{sec:basic}, one obtains
\[ {\bf z}{\cdot}{\bf z}' = \sum_{\rho\, \in \, \Irep(G)}\left( \sum_{g\in G} a^{(\ell)}_{(u,g)}\chi^\rho(g) \right) \chi^\rho(e) = \sum_{g\in G}a^{(\ell)}_{(u,g)}\left( \sum_{\rho\, \in \, \Irep(G)} \chi^\rho(g) \chi^\rho(e) \right) = a^{(\ell)}_{(u,e)}|G|\ . \]
On the other hand, realizing that $\chi^\rho(e)=d_\rho$ is the dimension of $\rho$, one has
\[ {\bf z}{\cdot}{\bf z}' = \sum_{\rho\, \in \, \Irep(G)} \chi^\rho((B^\ell)_{u,u}) \chi^\rho(e) = \sum_{\rho\, \in \, \Irep(G)} d_\rho \chi^\rho((B^\ell)_{u,u}), \]
and a comparison of the two results gives
\begin{equation}\label{eq:aell}
a^{(\ell)}_{(u,e)} = \frac{1}{|G|} \sum_{\rho\, \in \, \Irep(G)} d_\rho \chi^\rho((B^\ell)_{u,u})\ .
\end{equation}

The number of rooted closed walks of length $\ell$ in the regular lift $\Gamma^\alpha$ is obviously equal to the trace ${\rm tr}((A^\alpha)^\ell)$ of the $\ell$-th power of the adjacency matrix of the lift. By Lemma 1.1. of \cite{dfs19}, this trace (equal also to the sum of $\ell$-th powers of entries in the spectrum of the lift) can be evaluated as follows, where we used \eqref{eq:aell} in the last step:
\begin{equation}\label{eq:ellpow}
\tr ((A^{\alpha})^\ell)=\sum_{\lambda\in \Sp(\G^{\alpha})} \lambda^\ell=|G|\sum_{u\in V} a_{(u,e)}^{(\ell)} = \sum_{u\in V}\sum_{\rho\, \in\, \Irep(G)}d_\rho\chi^\rho((B^{\ell})_{u,u})\ .
\end{equation}
Developing now (\ref{eq:ellpow}) with the help of elementary facts about traces and group characters, and also recalling earlier definitions such as \eqref{eq:B} and \eqref{eq:chiB}, gives
\[\tr ((A^{\alpha})^\ell)
{=}\sum_{u\in V}\sum_{\rho\, \in\, \Irep(G)}d_\rho\chi^\rho((B^{\ell})_{u,u})
{=}\sum_{\rho\, \in\, \Irep(G)}d_\rho \sum_{u\in V} {\rm tr}(\rho(B^{\ell})_{u,u})
{=}\sum_{\rho\, \in\, \Irep(G)} d_\rho \tr(\rho(B^{\ell})),\]
and realizing the meaning of the left-hand side (see also (\ref{eq:ellpow})) and the right-hand side (which is simply the sum of all eigenvalues of $\rho(B^\ell)$, each taken $d_\rho$ times), one obtains
\begin{equation}
\label{eq:key}
\sum_{\lambda\in \Sp(\G^{\alpha})} \lambda^\ell = \sum_{\rho\, \in\, \Irep(G)} d_\rho \left(\sum_{\lambda\, \in\, {\rm Sp}(\rho(B^\ell ))} \lambda^\ell\right)\ .
\end{equation}
Notice that, in the sums on the right-hand side of \eqref{eq:key}, we have a total of $k\sum_{\rho\, \in\, \Irep(G)} d_\rho^2= k|G|$ terms, which is the number of eigenvalues of the adjacency matrix $A^{\alpha}$ of the digraph $\G^{\alpha}$ obtained as a regular lift of $\G$. Then, as the above equality holds for every $\ell=1,2,\ldots$, the multisets of eigenvalues on the left-hand and on the right-hand sides of \eqref{eq:key} must coincide (see, for example, Gould \cite{go99}). This furnishes an alternative proof of the `eigenvalues' part of Theorem \ref{t:spec} for regular lifts of digraphs.
\smallskip

As a consequence, in terms of group characters, the spectrum of a regular lift $\Gamma^{\alpha}$ of a digraph $\G$ with voltage group $G$ is determined as follows (see \cite{dfs19} for a full account).

\begin{proposition}[\cite{dfs19}]\label{propo-chi}
Let $P_{\ell}$ be the set of closed walks of length $\ell\ge 1$ in a connected digraph $\Gamma$ with vertex set  $V$. For a (closed) walk $p\in P_{\ell}$ of the form $u_0{\rightarrow} u_1{\rightarrow} \cdots {\rightarrow} u_{\ell-1}{\rightarrow} u_{\ell}{=}u_0$ and for a $\rho\in \Irep(G)$, we let $\chi^\rho(p)=\chi^\rho\left(\prod_{j=0}^{\ell-1}\alpha(u_ju_{j+1})\right)$.
Then, for each $\rho\in \Irep(G)$, the eigenvalues $\lambda_{u,j}$, for $u\in V$ and $j\in[d_\rho]$, of the lift $\G^{\alpha}$, are the solutions $($each repeated $d_\rho$ times$)$ of the system
	\begin{equation}
	\label{babai-gen}
	\mathop{\sum_{u\in V,}}_{j\in[d_\rho]} \lambda_{u,j}^{\ell}=\sum_{p\in P_{\ell}}\chi^\rho(p),\qquad \ell=1,\ldots, d_\rho|V|.
	\end{equation}
\end{proposition}
The right-hand side of \eqref{babai-gen} is equal to $\sum_{u\in V} \chi^\rho((B^{\ell})_{u,u})$, a quantity we may denote $\chi^\rho(\tr (B^{\ell}))$ by extending $\chi^\rho$ as done in \eqref{eq:chiB}. In particular, when $d_\rho=1$ for some $\rho$, in the notation of the proof of Theorem \ref{t:spec}, we simply have $\lambda_{u,1} =\mu_{u,1}$ for every $u\in V$. For instance, when $G$ is Abelian, this holds for every irreducible representation of $G$. This case was dealt with by Miller and Ryan together with the first two and the last authors in \cite{dfmrs17}.

\begin{example}
{\rm Consider the dumbbell graph of Figure \ref{fig1} again, with regular voltages in the symmetric group ${\rm Sym}(3)=\langle g, h \rangle$, and with the base matrix given by \eqref{base-matrix}:
$$
B=
\left(
\begin{array}{cc}
2g & e\\
e & 2h
\end{array}
\right).
$$
The corresponding regular lift $\Gamma_0^{\alpha}$ is shown in Figure \ref{fig2}. A set of irreducible characters for the group ${\rm Sym}(3)$ can be taken from traces of the representations $\iota,\pi,\sigma$ used in Section \ref{sec:example}; these are shown in Table \ref{charac-table-S3}.
\begin{table}[ht]
	\centering
	\begin{tabular}{|c||c|c|c|}
		\hline
		$S_3\quad \backslash\quad g$  & $e$  & $g,h,ghg$ & $gh,hg$ \\
		\hline\hline
		$\chi^\iota$ $(d_1=1)$  &  $1$  & $1$   &  $1$    \\
		\hline
		$\chi^\pi$ $(d_2=1)$  &  $1$  & $-1$   &  $1$   \\
		\hline
		$\chi^\sigma$ $(d_3=2)$  &  $2$  & $0$   &  $-1$    \\
		\hline
	\end{tabular}
	\caption{The character table of the symmetric group ${\rm Sym}(3)$.}
	\label{charac-table-S3}
\end{table}
To construct the spectrum of the regular lift $\Gamma_0^\alpha$ by Proposition~\ref{propo-chi}, it remains to evaluate the quantities $\chi^\rho(\tr(B^\ell))$ for our three characters and appropriate powers $\ell$, which gives the following.
\begin{itemize}
\item For $\chi^\iota$:
	Since $d_\iota=1$, the corresponding two eigenvalues of $\Gamma^{\alpha}$ are
    $$
	\{3,-1\}=\Sp (\chi^\iota (B)).
	$$
\item For $\chi^\pi$:
	Since $d_\pi=1$, the two corresponding eigenvalues of $\Gamma^{\alpha}$ are
	$$
	\{-3,1\}=\Sp (\chi^\pi (B)).
	$$
\item For $\chi^\sigma$:
	Since $d_\sigma=2$, we need to consider the traces of all the $\ell$-th powers of $B$ for  $\ell\in \{1,2,3,4{=}d_\sigma|V|\}$ in the group algebra $\C(G)$, which turn out to be
	\begin{align*}
	\tr(B) \  &=2 g + 2 h,\\
	\tr(B^2) &=4g^2+4h^2+2e=10e,\\
	\tr(B^3) &=8g^3+8h^3+6g+6h=14g+14h, \\
	\tr(B^4) &= 16g^4+16h^4+16g^2+16gh+16h^2+2e=66e+16gh.
	\end{align*}
	By Proposition~\ref{propo-chi} and the remark following it, the traces give the nonlinear system
	\begin{align*}
	\lambda_{u,0}+\lambda_{u,1}+\lambda_{v,0}+\lambda_{v,1} &=\chi^\sigma(\tr(B))=0,      \\
	\lambda_{u,0}^2+\lambda_{u,1}^2+\lambda_{v,0}^2+\lambda_{v,1}^2 &=\chi^\sigma(\tr(B^2))=20,\\
	\lambda_{u,0}^3+\lambda_{u,1}^3+\lambda_{v,0}^3+\lambda_{v,1}^3 &=\chi^\sigma(\tr(B^3))=0,\\
	\lambda_{u,0}^4+\lambda_{u,1}^4+\lambda_{v,0}^4+\lambda_{v,1}^4 &=\chi^\sigma(\tr(B^4))=116,
	\end{align*}
	with solution $\pm \sqrt{7}, \pm \sqrt{3}$. As these must be
	taken twice, the spectrum of $\Gamma_0^{\alpha}$ is
	$$
	\Sp(\Gamma_0^{\alpha})=\{3^{(1)},\sqrt{7}^{(2)},\sqrt{3}^{(2)},
	1^{(1)},-1^{(1)},-\sqrt{3}^{(2)},-\sqrt{7}^{(2)},-3^{(1)}\}.
	$$
	\end{itemize}
\vskip -2mm
\noindent Note that, as expected,
$\Sp(\Gamma)\subset \Sp(\Gamma^{\alpha})\subset \Sp(\Gamma_0^{\alpha})$. }
\end{example}

\begin{figure}[h!t]
	\begin{center}
		\scalebox{0.95}
		{\includegraphics{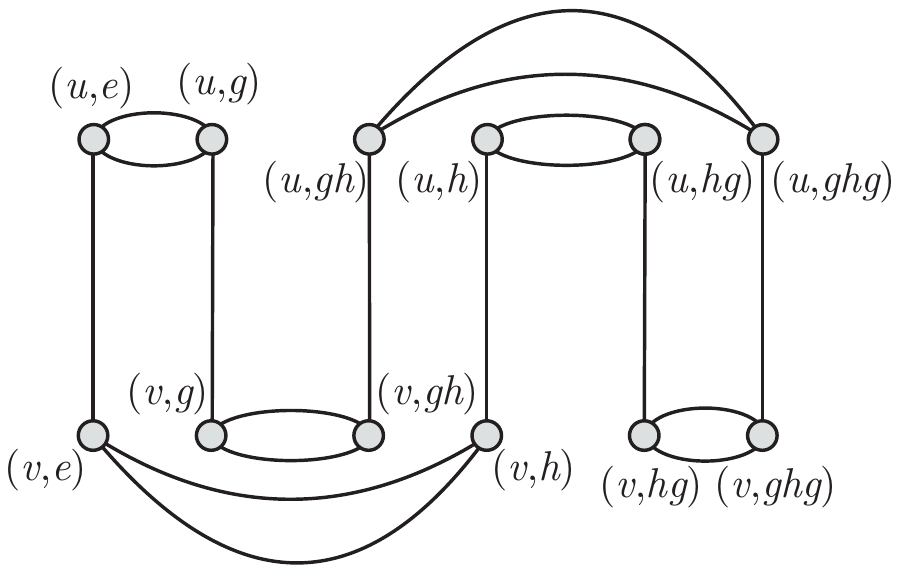}}
		\vskip -22.6cm
		\caption{The ordinary (regular) lift $\Gamma_0^{\alpha}$ of our base dumbbell graph $\Gamma$.
		\label{fig2}}
	\end{center}
\end{figure}

\noindent{\bf Acknowledgment}~ The research of the two first authors is partially supported by the project 2017SGR1087 of the Agency for the Management of University and Research Grants (AGAUR) of the Government of Catalonia. The third and the fourth authors acknowledge support from the APVV Research Grants 15-0220 and 17-0428, and the VEGA Research Grants 1/0142/17 and 1/0238/19.

\end{document}